\documentclass[a4wide,12pt]{amsart}  
\usepackage{amsmath,amssymb,amsthm,color}     
\usepackage[utf8]{inputenc}
\usepackage{enumitem}

\theoremstyle{plain}      
 
\newtheorem{thm}{Theorem}[section]     
\newtheorem{theorem}[thm]{Theorem}     
\newtheorem{cor}[thm]{Corollary}

\newtheorem{lemma}[thm]{Lemma}

\theoremstyle{definition}

\newtheorem{example}[thm]{Example}

\DeclareMathAlphabet{\doba}{U}{msb}{m}{n}

\gdef\mN{\doba{N}}

\gdef\mR{\doba{R}}

\gdef\mT{\doba{T}}
\gdef\mZ{\doba{Z}}

\def\GL{\mathrm{GL}}
\def\SL{\mathrm{SL}}

\def\di{{\rm d}}

\def\ie{\emph{i.e.} }

\def\h1{H^1_{\mathrm{dR}}(M)}

\let\<\langle 
\let\>\rangle

\newcommand{\definedas}{\mathrel{\raise.095ex\hbox{\rm :}\mkern-5.2mu=}}

\author{Andrei Moroianu and Mihaela Pilca}

\address{Andrei Moroianu \\ Université Paris-Saclay, CNRS,  Laboratoire de mathématiques d'Orsay, 91405, Orsay, France}
\email{andrei.moroianu@math.cnrs.fr}

\address{Mihaela Pilca\\Fakult\"at f\"ur Mathematik\\
Universit\"at Regensburg\\Universit\"atsstr. 31 
D-93040 Regensburg, Germany}
\email{mihaela.pilca@mathematik.uni-regensburg.de}

\subjclass[2010]{53A30, 53B35, 53C25, 53C29, 53C55}

\keywords{Twisted cohomology, Morse-Novikov cohomology}

\begin{document}   

\title{Closed $1$-forms and twisted cohomology}

\maketitle

\begin{center}{\it Dedicated to Paul Gauduchon on the occasion of his 75th birthday}\end{center}

\begin{abstract} We show that the first twisted cohomology group associated to closed $1$-forms on differentiable manifolds is related to certain $2$-dimensional representations of the fundamental group. In particular, we construct examples of nowhere-vanishing $1$-forms with non-trivial twisted cohomology.
\end{abstract}

\section{Introduction}

If $\theta$ is a closed $1$-form on a smooth manifold $M$, the {\em twisted differential} $\di_\theta:=\di-\theta\wedge$ maps $\Omega^k(M)$ to $\Omega^{k+1}(M)$ and satisfies $\di_\theta\circ\di_\theta=0$, thus defining the {\em twisted cohomology groups} 
$$H^k_\theta(M):=\frac{\ker(\di_\theta|_{\Omega^k(M)})}{\di_\theta(\Omega^{k-1}(M))}.$$

These groups only depend on the de Rham cohomology class of $\theta$, since the corresponding twisted differential complexes associated to cohomologous $1$-forms are canonically isomorphic. In particular, the twisted cohomology associated to an exact $1$-form is just the de Rham cohomology.

 It is well known that the twisted cohomology defined by the Lee form of Vaisman mani\-folds, and more generally by any non-zero $1$-form $\theta$ which is parallel with respect to some Riemannian metric on a compact manifold, vanishes \cite{llmp}. 
 
The twisted cohomology groups, as well as their Dolbeault and Bott-Chern counterparts, play an important role in locally conformally Kähler geometry ({\em cf}. \cite{n19} or \cite{ov}, where the twisted cohomology is called Morse-Novikov cohomology). 

Twisted cohomology was also used by A. Pajitnov \cite{p87}, who shows that if $\theta$ is a closed 1-form with non-degenerate zeros, then for large $t$ the dimension of $H^k_{t\theta}(M)$ gives a lower bound for the number of the zeros of $\theta$ of index $k$. 
This is an analog of Witten's approach to Morse theory, in the more general situation of closed 1-forms.

On the other hand, in \cite{p}, A. Pajitnov defined a different {\em twisted Novikov homology} theory associated to closed $1$-forms $\theta$ with integral cohomology class $[\theta]\in H^1(M,\mZ)$, and shows that the twisted Novikov homology vanishes whenever $[\theta]$ admits a nowhere-vanishing representative (\cite{p}, Theorem 1.3). We will see in Example \ref{ex} below that the corresponding result fails for the standard twisted cohomology theory considered here.

Our main result (Theorem \ref{pr}) relates the non-zero elements in the first twisted cohomology group associated to a closed $1$-form $\theta$ with some set of non-decomposable $2$-dimensional representations of the first fundamental group of $M$ which contain a trivial subrepresentation, and whose determinant is the character of $\pi_1(M)$ canonically associated to $\theta$.

In Section 3 we derive several applications of this result, like the vanishing of the first twisted cohomology group on manifolds with nilpotent fundamental group (Corollary \ref{nilp}), the fact that if the commutator group $[\pi_1(M),\pi_1(M)]$ is finitely generated, then the set $\{ [\theta]\in\h1\, |\, H^1_{\theta}(M)\neq 0\}$ is finite (Corollary \ref{fg}), or the non-vanishing of twisted cohomology on Riemann surfaces of genus $g\ge 2$ (Corollary \ref{rs}). In the last section we give several examples of explicit computations of the first twisted cohomology group on mapping tori or Vaisman manifolds.

{\sc Acknowledgments.} This work was supported by the Procope Project No. 57445459 (Germany) /  42513ZJ (France). We are grateful to Andrei Pajitnov and Chenji Fu for some very useful comments on earlier versions of this work.

\section{The Main Result}

Notation: the cohomology class of a $\di_\theta$-closed $1$-form $\alpha$ is denoted by $[\alpha]_\theta$.

Let us recall the following well-known result and present a proof for it, whose method will be useful in the sequel.
\begin{lemma}\label{h1dR}
	Let $M$ be a manifold. There is a bijection between 
	$$ H^1_{\mathrm{dR}}(M)\overset{1:1}{\longleftrightarrow} \{ \rho\colon \pi_1(M)\to (\mR^*_+,\times) \, |\, \rho \text{ is a representation}\}.$$
\end{lemma}	

\begin{proof}
	Let $\theta$ be a representative of a cohomology class $[\theta]\in H^1_{\mathrm{dR}}(M)$ and denote the universal cover of $M$ by $\pi:\widetilde M\to M$. Then the pull-back $\widetilde\theta:=\pi^*\theta$ of $\theta$ is an exact form, \ie there exists $\varphi\in\mathcal{C}^\infty(\widetilde M)$ such that $\widetilde\theta=\di\varphi$. Any element $\gamma\in\pi_1(M)$ acts trivially on $\widetilde\theta$, so $\gamma^*\di\varphi=\di\varphi$, which implies the existence of a constant $c_\gamma\in\mR$ with $\gamma^*\varphi=\varphi+c_{\gamma}$. Since $\gamma_1^*\gamma_2^*=(\gamma_2\gamma_1)^*$, we see that $\gamma\mapsto c_\gamma$ is a group morphism from $\pi_1(M)$ to $(\mR,+)$.  We then associate to $[\theta]\in H^1_{\mathrm{dR}}(M)$ the representation $\rho\colon \pi_1(M)\to (\mR^*_+,\times)$ defined by $\rho(\gamma):=e^{c_\gamma}$. The representation $\rho$ does not depend on the choice of the representative $\theta$ in its cohomology class. Indeed, if we replace $\theta$ by $\theta+\di h$, then $\varphi$ is replaced by $\varphi+\pi^*h$, and since $\pi^*h$ is invariant by $\pi_1(M)$, the constants $c_\gamma$ do not change.
	
Conversely, for any representation $\rho\colon \pi_1(M)\to (\mR^*_+,\times)$ we will construct a positive function $g$ on $\widetilde M$ which is $\rho$-equivariant, i.e. $a^*g=\rho(a)g$ for every $a\in\pi_1(M)$. To do this, let us pick a non-negative function $f$ on $\widetilde M$ satisfying the properties (i) and (ii) of Lemma \ref{function} below. We introduce the function $$g:=\displaystyle \sum_{\gamma\in\pi_1(M)}\rho(\gamma^{-1})\gamma^*f$$ which is well-defined and smooth on $\widetilde M$ since the sum is finite in the neighbourhood of any point of $\widetilde M$ by property (ii). Moreover, $g$ is a positive function on $\widetilde M$ since $f>0$ on $V$ and $\pi_1(M)\cdot V=\widetilde M$ by property (i). For any $a\in\pi_1(M)$, we have: 
	\[ a^*g=\sum_{\gamma\in\pi_1(M)}\rho(\gamma^{-1})(\gamma a)^*f=\sum_{\delta\in\pi_1(M)}\rho(a\delta^{-1})\delta^*f=\rho(a)g.\]
	
	This shows that $\widetilde\theta:=\di(\ln g)$ is an exact $1$-form on $\widetilde M$, which is $\pi_1(M)$-invariant, hence $\widetilde\theta$ descends to a closed $1$-form $\theta$ on $M$. We associate to $\rho$ the cohomology class of $\theta$ in $H^1_{\mathrm{dR}}(M)$. This does not depend on the choice of $f$. Indeed, if $g_1$ is any other positive function on $\widetilde M$ satisfying $ a^*g=\rho(a)g$ for every $a\in\pi_1(M)$, then $g_1/g$ is $\pi_1(M)$-invariant, so it is the pull-back to $\widetilde M$ of some function $h$ on $M$. Then the closed $1$-form $\theta_1$ on $M$ satisfying $\pi^*\theta_1=\di(\ln g_1)$ is $\theta_1=\theta+\di h$, so $[\theta_1]=[\theta].$
	
	One can easily check that the above defined maps are inverse to each other.

\end{proof}

\begin{lemma}\label{function} There exists a non-negative function $f\in\mathcal{C}^{\infty}(\widetilde M, \mR_+)$ satisfying the following properties:
\begin{enumerate}
\item[(i)] $f$ is positive on some open set $V\subset \widetilde M$ with $\pi_1(M)\cdot V=\widetilde M$;
\item[(ii)] any point $x\in\widetilde M$ has an open neighborhood $V_x$, such that the set 
$$\{\gamma\in \pi_1(M)\, |\, \gamma\cdot V_x \cap \mathrm{supp}(f)\neq \emptyset  \}$$ is finite.
\end{enumerate}
\end{lemma}
\begin{proof}
Denote by $\pi:\widetilde M\to M$ the covering map and let $(U_i)_{i\in I}$ be an open cover of $M$ with contractible open sets. Since $U_i$ are simply connected, there exist open sets $V_i$ of $\widetilde M$ such that $\pi|_{V_i}:V_i\to U_i$ is a diffeomeorphism for each $i\in I$. 

Let $(\rho_i)_{i\in I}$ be a partition of unity subordinate to the open cover $(U_i)_{i\in I}$. By definition, we have $\rho_i\ge 0$, $\mathrm{supp}(\rho_i)\subset U_i$, and every point $y\in M$ has an open neighbourhood $U_y$ such that the set 
\begin{equation}\label{iy}I_y:=\{i\in I\, |\, U_y \cap \mathrm{supp}(\rho_i)\neq \emptyset  \}\end{equation} 
is finite.
We define $f_i:\widetilde M\to \mathbb{R}_+$ by $f_i|_{V_i}=\rho_i\circ\pi$ and $f_i|_{\widetilde M\setminus V_i}=0$. Clearly $f_i$ are smooth since $\mathrm{supp}(\rho_i)\subset U_i$. For every point $x\in \widetilde M$, there is only a finite number of $i\in I$ for which the open set $\pi^{-1}(U_{\pi(x)})$ meets the support of $f_i$, so the function $f:=\sum_{i\in I}f_i$ is well-defined, smooth and non-negative on $\widetilde M$. We claim that it also satisfies the properties (i) and (ii). 

Let $V:=f^{-1}(\mR^*_+)$ be the open set where $f$ is positive. For every $x\in \widetilde M$, there exists $i\in I$ such that $\rho_i(\pi(x))>0$. If $\gamma$ denotes the unique element in $\pi_1(M)$ for which $\gamma(x)\in V_i$, then $f_i(\gamma(x))>0$, so $f(\gamma(x))>0$, showing that $x\in \gamma^{-1}(V)$. Thus $\pi_1(M)\cdot V=\widetilde M$, so (i) is verified.

Let now $x\in \widetilde M$ be any point. We define 
$$V_x:=\bigcup_{i\in I_{\pi(x)}}V_i,$$
where $I_{\pi(x)}$ is the finite subset of $I$ given by \eqref{iy}. Then $\{\gamma\in \pi_1(M)\, |\, \gamma\cdot V_x \cap \mathrm{supp}(f_i)\neq \emptyset  \}$ is empty for every $i\in I\setminus I_{\pi(x)}$ and has exactly one element for every $i\in I_{\pi(x)}$. This shows that the set of $\gamma\in \pi_1(M)$ for which $\gamma\cdot V_x$ meets the support of $f$ is finite, having the same cardinal as $I_{\pi(x)}$. 

\end{proof}

\begin{theorem}\label{pr}
	Let $M$ be a manifold and let $\theta$ be some non-exact closed $1$-form on $M$. Let $\rho\colon \pi_1(M)\to (\mR_+^*,\times)$ denote the representation associated to $[\theta]\in\h1$, as in Lemma~\ref{h1dR}. Then the following assertions hold:
	\begin{enumerate}
		\item  
		If $H^1_\theta(M)\neq 0$, then there exists an indecomposable representation $\xi\colon \pi_1(M)\to \GL_2(\mR)$ with $\det\xi=\rho$, which fixes the vector $\begin{pmatrix} 1 \\ 0\end{pmatrix}\in\mR^2$.
		\item Conversely, if there exists an indecomposable representation $\xi\colon \pi_1(M)\to \GL_2(\mR)$ with $\det\xi=\rho$ and which fixes the vector $\begin{pmatrix} 1 \\ 0\end{pmatrix}\in\mR^2$, then $H^1_\theta(M)\neq 0$.
	\end{enumerate}
\end{theorem}	

\begin{proof}
	$(1)$ Let $\alpha$ be a $\di_\theta$-closed $1$-form on $M$ whose twisted cohomology class $[\alpha]_\theta\in H^1_{\theta}(M)$ is non-zero: $[\alpha]_\theta\neq 0$. If $\pi\colon\widetilde M\to M$ denotes as before the universal cover map and $\varphi$ is a primitive of $\pi^*\theta$ on $\widetilde M$, then 
	\begin{equation}\label{conj}\pi^*\di_\theta=e^{\varphi}\di e^{-\varphi}\pi^*, 
	\end{equation}
	so that $\di_\theta\alpha=0$ is equivalent to $\di (e^{-\varphi}\pi^*\alpha)=0$ on $\widetilde M$. Hence, there exists a function $h\in\mathcal{C}^\infty(\widetilde M)$, such that $e^{-\varphi}\pi^*\alpha=\di h$, and thus $\gamma^*(\di h)=e^{-c_\gamma}\di h=\rho(\gamma^{-1})\di h$. Therefore, there exists for each $\gamma\in\pi_1(M)$ a constant $\lambda(\gamma)\in\mR$, such that 
	$$\gamma^*h=\rho(\gamma^{-1})h+\lambda(\gamma),$$
	which equivalently reads
	\begin{equation}\label{gammah}
	(\gamma^{-1})^*h=\rho(\gamma)h+\lambda(\gamma^{-1}), \qquad\gamma\in\pi_1(M).
	\end{equation}
	
	We claim that the map $\xi\colon \pi_1(M)\to \GL_2(\mR)$ defined by
\begin{equation}\label{xi}\xi(\gamma):=\begin{pmatrix} 1 & \lambda(\gamma^{-1})\\ 0 & \rho(\gamma)\end{pmatrix}
\end{equation}
	is a group morphism. Indeed, if $\gamma_1,\gamma_2\in\pi_1(M)$, we have by \eqref{gammah}:
	$$((\gamma_1\gamma_2)^{-1})^*h=(\gamma_1^{-1})^*(\gamma_2^{-1})^*h=(\gamma_1^{-1})^*(\rho(\gamma_2)h+\lambda(\gamma_2^{-1}))=\rho(\gamma_2)(\rho(\gamma_1)h+\lambda(\gamma_1^{-1}))+\lambda(\gamma_2^{-1}),$$
	thus showing that 
	$\rho(\gamma_1\gamma_2)=\rho(\gamma_1)\rho(\gamma_2)$ and $\lambda((\gamma_1\gamma_2)^{-1})=\rho(\gamma_2)\lambda(\gamma_1^{-1})+\lambda(\gamma_2^{-1}).$ Consequently,
	$$\xi(\gamma_1)\xi(\gamma_2)=\begin{pmatrix} 1 & \lambda(\gamma_1^{-1})\\ 0 & \rho(\gamma_1)\end{pmatrix}\begin{pmatrix} 1 & \lambda(\gamma_2^{-1})\\ 0 & \rho(\gamma_2)\end{pmatrix}=\begin{pmatrix} 1 & \rho(\gamma_2)\lambda(\gamma_1^{-1})+\lambda(\gamma_2^{-1})\\ 0 & \rho(\gamma_1) \rho(\gamma_2)\end{pmatrix}=\xi(\gamma_1\gamma_2).$$

We clearly have that $\det(\xi)=\rho$. It remains to check that $\xi$ is indecomposable. Assuming by contradiction that there exists a one-dimensional subrepresentation $V\subset \mR^2$ of $\xi$ with $V\neq \left\langle \begin{pmatrix} 1 \\ 0 \end{pmatrix}\right\rangle$, then $V$ is generated by some vector $\begin{pmatrix} c \\ 1 \end{pmatrix}\in\mR^2$. By \eqref{xi}, for each $\gamma\in\pi_1(M)$ we have 
$$\xi(\gamma)\begin{pmatrix} c \\ 1 \end{pmatrix}=\begin{pmatrix} c+ \lambda(\gamma^{-1})\\ \rho(\gamma) \end{pmatrix}.$$
Thus $V$ is preserved by $\xi$ if and only if $\lambda(\gamma^{-1})+c=\rho(\gamma)c$ for every $\gamma\in\pi_1(M)$.

 Together with \eqref{gammah} we obtain:
	$$ (\gamma^{-1})^*(h+c)=(\gamma^{-1})^*h+c=\rho(\gamma)h+\lambda(\gamma^{-1})+c=\rho(\gamma)h+\rho(\gamma)c=\rho(\gamma)(h+c).$$
	This shows that $e^{\varphi} (h+c)$ is the pull-back through $\pi$ of a function on $M$, \ie there exists $s\in\mathcal{C}^{\infty}(M)$ such that $h+c=e^{-\varphi}\pi^* s$. However, this yields:
	$$ e^{-\varphi}\pi^*\alpha=\di h=\di (h+c)=\di(e^{-\varphi}\pi^*s)=e^{-\varphi}\pi^*\di_{\theta}s,$$
whence $\alpha=\di_{\theta} s$, contradicting that $[\alpha]_\theta\neq 0$. We thus conclude that $\xi$ is indecomposable.\\
	
	$(2)$ 
	We denote by $M_\gamma$ the matrix of $\xi(\gamma)$ with respect to the standard basis $\left\{\begin{pmatrix} 1 \\ 0\end{pmatrix}, \begin{pmatrix}  0\\ 1\end{pmatrix}\right\}$, which is of the form $M_\gamma=\begin{pmatrix} 1 & \lambda(\gamma^{-1})\\ 0 & \rho(\gamma) \end{pmatrix}$.
Consider again the function $f\in\mathcal{C}^\infty(\widetilde M, \mR_+)$ given by Lemma \ref{function}, and define the function $g\colon \widetilde M\to \mR^2$ as follows: $$g=\begin{pmatrix} g_1 \\ g_2	\end{pmatrix}:=\sum_{\gamma\in\pi_1(M)} M_{\gamma^{-1}}\cdot \gamma^*\begin{pmatrix} 0 \\ f \end{pmatrix}.$$
As before, the function $g$ is well-defined and smooth, since the sum is finite in the neighbourhood of any point of $\widetilde{M}$, by property (ii) in Lemma \ref{function}. Note that the function $g_2=\displaystyle \sum_{\gamma\in\pi_1(M)} \rho(\gamma^{-1}) \gamma^*f$ is positive on $\widetilde M$, by property (i) in Lemma \ref{function}.
	We compute for any $a\in \pi_1(M)$:
	\begin{equation*}
	\begin{split}
	a^*g&=\sum_{\gamma\in\pi_1(M)} M_{\gamma^{-1}}\cdot a^*\gamma^*\begin{pmatrix} 0 \\ f \end{pmatrix}=\sum_{\gamma\in\pi_1(M)} M_{\gamma^{-1}}\cdot (\gamma a)^*\begin{pmatrix} 0 \\ f \end{pmatrix}\\
	&=\sum_{\gamma\in\pi_1(M)} M_{a\gamma^{-1}}\cdot \gamma^*\begin{pmatrix} 0 \\ f \end{pmatrix}=M_a\cdot\sum_{\gamma\in\pi_1(M)} M_{\gamma^{-1}}\cdot \gamma^*\begin{pmatrix} 0 \\ f \end{pmatrix}=M_a\cdot g.
	\end{split}
	\end{equation*}
	Thus,  for any $a\in \pi_1(M)$, we have:
	$$\begin{pmatrix}  a^*g_1\\a^* g_2\end{pmatrix}=\begin{pmatrix}  g_1+\lambda(a^{-1})g_2\\\rho(a) g_2\end{pmatrix}.$$
	Since $g_2>0$ on $\widetilde M$ and satisfies $a^* g_2=\rho(a)g_2$, for all $a\in\pi_1(M)$, we conclude as in the proof of Lemma \ref{h1dR}, that $\di(\ln g_2)$ is the pull-back of a closed $1$-form $\theta'$ on $M$ cohomologous to $\theta$. Up to changing the representative, we may assume that $\pi^*\theta=\di(\ln g_2)$.
	
	We define $h\colon \widetilde M\to \mR$, $h:=\frac{g_1}{g_2}$ and compute for every $a\in\pi_1(M)$:
	\begin{equation}\label{ah}
	\begin{split}
	a^* h=\frac{a^*g_1}{a^*g_2}=\frac{g_1+\lambda(a^{-1})g_2}{\rho(a)g_2}=\rho(a^{-1})h+\rho(a^{-1})\lambda(a^{-1}).
	\end{split}
	\end{equation}
	This shows that $a^*\di h=\rho(a^{-1})\di h$ for all $a\in\pi_1(M)$, so the $1$-form $g_2\di h$ is invariant under the action of $\pi_1(M)$. Consequently, there exists $\alpha\in \Omega^1(M)$ with $\pi^*\alpha=g_2\di h$. We now check that $\alpha$ defines a non-trivial twisted cohomology class in $H^1_\theta(M)$. Firstly, $\alpha$ is $\di_{\theta}$ closed, because
	$$\pi^*(\di_{\theta}\alpha)=e^{\varphi}\di e^{-\varphi}\pi^*\alpha=g_2\di\left( \frac{1}{g_2}\pi^*\alpha\right)=g_2\di(\di h)=0.$$
	We now assume that $[\alpha]_\theta=0$ in $H^1_\theta(M)$, \ie there exists $s\in\mathcal{C}^\infty(M)$ such that $\alpha=\di_{\theta} s$. Using \eqref{conj}, this implies
	$$g_2\di h=\pi^*\alpha=\pi^*\di_{\theta} s=g_2\di \left(\frac{1}{g_2}\pi^* s\right),$$
	hence there exists a constant $c$ such that $h=\frac{1}{g_2}\pi^* s+c$. We claim that the one-dimensional eigenspace spanned by the vector $\begin{pmatrix} c \\1 \end{pmatrix}\in\mR^2$ is invariant under $\xi$. Namely, the following equality holds for all $a\in\pi_1(M)$, according to \eqref{ah} and to the definition of $c$:
	$$\rho(a^{-1})h+\rho(a^{-1})\lambda(a^{-1})=a^*h=a^*\left(\frac{1}{g_2}\pi^* s+c\right)=\frac{\pi^*s}{\rho(a)g_2}+c=\rho(a^{-1})(h-c)+c,$$
	which implies that $c+\lambda(a^{-1})=\rho(a)c$. Hence, for any $a\in\pi_1(M)$, we have:
	\[\xi(a)\begin{pmatrix} c \\ 1\end{pmatrix}=M_a\begin{pmatrix} c \\ 1\end{pmatrix}=\begin{pmatrix} 1 & \lambda(a^{-1})\\ 0 & \rho(a)\end{pmatrix}\begin{pmatrix} c \\ 1\end{pmatrix}=\begin{pmatrix} c+\lambda(a^{-1}) \\ \rho(a)\end{pmatrix}=\rho(a)\begin{pmatrix} c \\ 1\end{pmatrix}.\]
	This contradicts the assumption that $\xi$ is indecomposable, hence we conclude that $[\alpha]_\theta\neq 0$. 
	
\end{proof}	

The indecomposability hypothesis in the above result can be equivalently stated as follows:

\begin{lemma}\label{dec} 
Let $\xi:\Gamma\to \GL_2(\mR)$ be a two-dimensional representation of a group $\Gamma$, which fixes the vector $\begin{pmatrix}1\\0\end{pmatrix}\in\mR^2$ and such that ${\rho}:=\det(\xi)$ is non-trivial. Then $\xi$ is decomposable if and only if $[\Gamma,\Gamma]\subset\ker(\xi)$.
\end{lemma}

\begin{proof} If $\xi$ is decomposable, then all matrices in $\xi(\Gamma)$ are simultaneously diagonalizable, so they commute, whence $\xi([\Gamma,\Gamma])=\{\rm{I}_2\}$. 

Assume, conversely, that $[\Gamma,\Gamma]\subset\ker(\xi)$. By hypothesis, there exists some $\gamma_0\in\Gamma$ with $\rho(\gamma_0)\ne 1$. Then $\xi(\gamma_0)$ has two distinct eigenvalues, $1$ and ${\rho(\gamma_0)}$, so it has two one-dimensional eigenspaces $E_1$ and $E_2$. For every element $\gamma\in\Gamma$, $\xi(\gamma)$ commutes with $\xi(\gamma_0)$, so $\xi(\gamma)$ preserves $E_1$ and $E_2$. Thus $\xi$ is decomposable.
\end{proof}

\section{Applications}

We now derive some consequences of Theorem \ref{pr}.

\begin{cor}\label{nilp}
	Let $M$ be a manifold whose fundamental group $\pi_1(M)$ is nilpotent. Then for any non-trivial cohomology class $[\theta]\in\h1$, we have $H^1_{\theta}(M)=0$. 
\end{cor}

\begin{proof}
	Let $[\theta]\in\h1$ with $[\theta]\neq 0$, and let $\rho\colon \pi_1(M)\to (\mR_+^*,\times)$ denote the representation associated to $[\theta]\in\h1$, given by Lemma~\ref{h1dR}.
	Applying Theorem \ref{pr}, we have to show that any representation $\xi\colon \pi_1(M)\to \GL_2(\mR)$ with $\det\xi=\rho$ and which fixes the vector $\begin{pmatrix} 1\\0 \end{pmatrix}$ is decomposable. We assume by contradiction that there exists such a representation $\xi$ which is indecomposable. 
	
	Since $[\theta]\neq 0$, we have $\rho\neq 1$, so there exists $a\in\pi_1(M)$ such that $\det(\xi(a))\neq 1$. Then $\xi(a)$ is diagonalizable, so there exists a basis of $\mR^2$, such that the matrix of $\xi(a)$ with respect to this basis is given by $M_a=\begin{pmatrix} 1 & 0 \\ 0 & \rho(a) \end{pmatrix}$. Since $\xi$ is assumed to be indecomposable, by Lemma \ref{dec}, there exists $b_0\in[\pi_1(M), \pi_1(M)]$
	with $M_{b_0}=\begin{pmatrix} 1 & \lambda(b_0^{-1})\\
	0 & \rho(b_0) \end{pmatrix}$ and $\lambda(b_0)\neq 0$. We then obtain for $b_1:=b_0^{-1}a^{-1}b_0a$:
	\begin{equation*}
	\begin{split}
	M_{b_1}&=\begin{pmatrix} 1 & -\frac{\lambda(b_0^{-1})}{\rho(b_0)}\\ 0 & \frac{1}{\rho(b_0)}\end{pmatrix}\begin{pmatrix} 1 & -\frac{\lambda(a^{-1})}{\rho(a)}\\ 0 & \frac{1}{\rho(a)}\end{pmatrix}\begin{pmatrix} 1 & \lambda(b_0^{-1})\\ 0 & \rho(b_0)\end{pmatrix}\begin{pmatrix} 1 & \lambda(a^{-1})\\ 0 & \rho(a)\end{pmatrix}\\[0.2cm]
	&=\begin{pmatrix} 1 & \lambda(b_0^{-1})(\rho(a)-1)\\ 0 & 1\end{pmatrix},
	\end{split}
	\end{equation*}
	which shows that also $\lambda(b_1^{-1})=\lambda(b_0^{-1})(\rho(a)-1)\neq 0$, because $\rho(a)\neq 1$ and $\lambda(b_0)\neq 0$. If we define for $i\in\mN$ inductively $b_{i+1}:=b_i^{-1}a_0^{-1}b_ia_0$, then $\lambda(b_i)\neq 0$, for all $i$, which contradicts the hypothesis that $\pi_1(M)$ is nilpotent.
	
\end{proof}

\begin{cor}\label{fg}
	Let $M$ be a manifold whose commutator subgroup $G:=[\pi_1(M), \pi_1(M)]$ is finitely generated. Then the set 
	$$\{ [\theta]\in\h1\, |\, H^1_{\theta}(M)\neq 0\}$$
	is finite and has at most $\displaystyle\mathrm{rank}(G)^{\mathrm{rank}(\pi_1(M))}$  elements.
\end{cor}

\begin{proof}
	Let $\{a_1, \dots, a_m\}$ be a set of generators of $\pi_1(M)$ and let $\{b_1, \dots, b_k\}$ be a set of generators of $G$. Let $[\theta]\in \h1$  with $H^1_{\theta}(M)\neq 0$. Let  $\rho\colon \pi_1(M)\to (\mR_+^*,\times)$ denote the representation associated to $[\theta]\in\h1$, given by Lemma~\ref{h1dR}, and let $\xi\colon \pi_1(M)\to \GL_2(\mR)$ be a representation associated to $[\theta]$, as in Theorem \ref{pr}. 
	
	We denote by $M_i$ the matrix of $\xi(b_i)$ with respect to the standard basis of $\mR^2$. Since $b_i\in G=[\pi_1(M), \pi_1(M)]$, we have $\rho(b_i)=1$, so the matrix $M_i$ has the following form: $M_i=\begin{pmatrix} 1 & x_i\\ 0 & 1\end{pmatrix}$, for some $x_i\in\mR$. Let us remark that at least one of the numbers $x_i$ does not vanish, since otherwise the restriction of $\xi$ to $G$ would be trivial and then, by Lemma \ref{dec}, $\xi$ would be decomposable. 
	
	For any $1\leq j\leq m$ and $1\leq i\leq k$, the element $a_j^{-1}b_ia_j$ belongs to $G$. Therefore, there exist  integers $n_{ij\ell}$, for $1\leq \ell\leq k$, such that $a_j^{-1}b_ia_j=\displaystyle\prod_{\ell=1}^k b_{\ell}^{n_{ij\ell}}$. On the one hand, we compute:
	\begin{equation*}
	\begin{split}
	\xi(a_j^{-1})\xi(b_i)\xi(a_j)=\begin{pmatrix} 1 & -\frac{\lambda(a_j^{-1})}{\rho(a_j)}\\[0.2cm] 0 & \frac{1}{\rho(a_j)}\end{pmatrix}\begin{pmatrix} 1 & x_i\\ 0 & 1\end{pmatrix}\begin{pmatrix} 1 & \lambda(a_j^{-1})\\ 0 & \rho(a_j)\end{pmatrix}=\begin{pmatrix} 1 & x_i\rho(a_j)\\ 0 & 1\end{pmatrix}.
	\end{split}
	\end{equation*}
	On the other hand, we have:
	\begin{equation*}
	\begin{split}
	\xi(a_j^{-1})\xi(b_i)\xi(a_j)=\xi(a_j^{-1}b_i a_j)=\xi\left(\prod_{\ell=1}^k b_{\ell}^{n_{ij\ell}}\right)= \prod_{\ell=1}^k M_{\ell}^{n_{ij\ell}} =\begin{pmatrix} 1 & \displaystyle \sum_{\ell=1}^k n_{ij\ell}x_\ell\\[0.2cm] 0 & 1\end{pmatrix}.
	\end{split}
	\end{equation*}
	Hence,  for all $1\leq j\leq m$ and $1\leq i\leq k$, the following equality holds: $x_i\rho(a_j)=\displaystyle\sum_{\ell=1}^{k}n_{ij\ell}x_\ell$.  If for each fixed $j\in\{1, \dots, m\}$,  we define the $k\times k$-matrix with integer entries $N_j:=(n_{ij\ell})_{i,\ell}$, then the above system of equations for $j$ fixed can be equivalently written as: $$(N_j-\rho(a_j){\rm I}_k)\begin{pmatrix} x_1 \\
	\vdots\\
	x_k \end{pmatrix}=0.$$ 
	
	As previously noticed, at least one of the $x_i$'s is non-zero. Thus $\rho(a_j)$ must be an eigenvalue of $N_j$, so each $\rho(a_j)$ can take at most $k$ different values. Therefore, when $j$ varies, there are overall at most $k^m$ different possibilities for defining $\rho$, or, equivalently, for defining a cohomology class $[\theta]\in \h1$ with $H^1_{\theta}(M)\neq 0$.
	
\end{proof}	

\begin{cor}\label{rs} If $S$ is a compact Riemann surface of genus $g\ge 2$, then $H^1_\theta(S)\ne 0$ for every closed $1$-form $\theta$ on $S$. 
\end{cor}

\begin{proof} It is well known that $\pi_1(S)$ has $2g$ generators $\gamma_1,\ldots,\gamma_{2g}$ subject to the relation 
\begin{equation}\label{rel}\prod_{j=1}^g(\gamma_{2j-1}\gamma_{2j}\gamma_{2j-1}^{-1}\gamma_{2j}^{-1})=1.
\end{equation}

Any representation $\rho:\pi_1(S)\to (\mR^*_+,\times)$ is defined by the $2g$ positive real numbers $y_i:=\rho(\gamma_i)$.
According to Lemma \ref{dec} and Theorem \ref{pr}, we need to show for every such $\rho$, there exists a two-dimensional representation $\xi:\pi_1(S)\to \GL_2(\mR)$ with $\det(\xi)={\rho}$, which fixes the vector $\begin{pmatrix}1\\0\end{pmatrix}\in\mR^2$ and whose restriction to $[\pi_1(S),\pi_1(S)]$ is non-trivial. 

We look for $\xi$ of the form $\xi(\gamma_i):=\begin{pmatrix}1&x_i\\0&y_i\end{pmatrix}.$ The commutator of two such matrices is 
$$\begin{pmatrix}1&x_i\\0&y_i\end{pmatrix}\begin{pmatrix}1&x_j\\0&y_j\end{pmatrix}\begin{pmatrix}1&x_i\\0&y_i\end{pmatrix}^{-1}\begin{pmatrix}1&x_j\\0&y_j\end{pmatrix}^{-1}=\begin{pmatrix}1&x_i(y_j-1)-x_j(y_i-1)\\0&1\end{pmatrix},$$
so by \eqref{rel}, the condition that $\xi$ defines a representation reads
\begin{equation}\label{c1}\sum_{j=1}^g\left(x_{2j-1}(y_{2j}-1)-x_{2j}(y_{2j-1}-1)\right)=0.\end{equation}

Moreover, such a representation is non-trivial on $[\pi_1(S),\pi_1(S)]$ provided that 
\begin{equation}\label{c2}\exists\  i,j\in\{1,\ldots, 2g\}\ \hbox{ such that}\ x_i(y_j-1)-x_j(y_i-1)\ne 0. \end{equation}
Since $g\ge 2$, for any positive real numbers $y_i$ ($1\le i\le 2g$), one can choose the real numbers $x_i$ such that \eqref{c1} and \eqref{c2} are satisfied. 

\end{proof}

\section{Examples}

Let $f_A$ be the diffeomorphism of the torus $\mT^2=\mR^2/\mZ^2$ induced by a matrix $A\in\SL_2(\mZ)$ and let $M_A$ be the mapping torus of $f_A$. In other words, $M_A$ is the quotient of $\mT^2\times\mR$ by the free $\mZ$-action generated by the diffeomorphism $(p,t)\mapsto (f_A(p),t+1)$. The fundamental group of $M_A$ is isomorphic to the semidirect product of $\mZ$ acting on $\mZ^2$: $\pi_1(M_A)\simeq \mZ^2\rtimes_A \mZ$. 

We pick some non-zero constant $c\in\mR$ and denote by $\theta_c$ the closed form on $M_A$ whose pull-back to $\mT^2\times\mR$ is $c\,\di t$. The associated representation $\rho_c:\pi_1(M_A)\to (\mR^*_+,\times)$ maps $\mZ^2$ to $1$ and the generator of $\mZ$ to $e^c$.

\begin{lemma}\label{ec} $H^1_{\theta_c}(M_A)\ne 0$ if and only if $e^c$ is an eigenvalue of $A$.
\end{lemma}

\begin{proof}
If $H^1_{\theta_c}(M_A)\ne 0$, Theorem \ref{pr} shows that there exists an indecomposable representation $\xi:\pi_1(M_A)\to \GL_2(\mR)$ which fixes the vector $\begin{pmatrix}1\\0\end{pmatrix}\in\mR^2$ and such that $\det(\xi)={\rho_c}$. This means that for every $v\in \mZ^2$ there exists $\lambda(v)\in\mR$ such that $\xi(v)=\begin{pmatrix}1&\lambda(v)\\0&1\end{pmatrix}$ and if $a$ denotes the generator of the subgroup $\mZ\subset \pi_1(M_A)$, there exists $x\in\mR$ such that $\xi(a)=\begin{pmatrix}1&x\\0&e^c\end{pmatrix}$. 

The map $\lambda$ is clearly a group morphism from $\mZ^2$ to $(\mR,+)$, so 
\begin{equation}\label{la1}\lambda(v_1,v_2)=\lambda_1v_1+\lambda_2v_2,\qquad\forall v=(v_1,v_2)\in\mZ^2. 
\end{equation}

Moreover, by Lemma \ref{dec}, $\lambda$ is not identically zero since $[\pi_1(M_A),\pi_1(M_A)]=\mZ^2$.

Since $ava^{-1}=Av$, we get 
$$\begin{pmatrix}1&\lambda(Av)\\0&1\end{pmatrix}=\begin{pmatrix}1&x\\0&e^c\end{pmatrix}\begin{pmatrix}1&\lambda(v)\\0&1\end{pmatrix}\begin{pmatrix}1&x\\0&e^c\end{pmatrix}^{-1}=\begin{pmatrix}1&e^{-c}\lambda(v)\\0&1\end{pmatrix},$$
whence
\begin{equation}\label{la2}\lambda(Av)=e^{-c}\lambda(v),\qquad\forall v\in\mZ^2.
\end{equation}

By \eqref{la1}, this is equivalent to 
\begin{equation}\label{la3}^t\! A\begin{pmatrix}\lambda_1\\ \lambda_2\end{pmatrix}=e^{-c}\begin{pmatrix}\lambda_1\\ \lambda_2\end{pmatrix}.
\end{equation}

Thus $e^{-c}$ is an eigenvalue of $^t\! A$, and since the spectra of $A$ and $^t\! A$ are the same and $\det(A)=1$, it follows that $e^c$ is an eigenvalue of $A$.

Conversely, if $e^c$ is an eigenvalue of $A$, then there exists $(\lambda_1,\lambda_2)\in\mR^2\setminus\{0\}$ such that \eqref{la3} holds. Then \eqref{la2} also holds for $\lambda$ defined by \eqref{la1}. 

We can then define a representation $\xi:\pi_1(M_A)\simeq \mZ^2\rtimes_A \mZ\to \GL_2(\mR)$ by $\xi(v):=\begin{pmatrix}1&\lambda(v)\\0&1\end{pmatrix}$, for $v\in\mZ^2$ and $\xi(k):=\begin{pmatrix}1&0\\0&e^{ck}\end{pmatrix}$, for $k\in \mZ$. By Lemma \ref{dec}, this representation is indecomposable, so by Theorem \ref{pr}, we conclude that $H^1_{\theta_c}(M_A)\ne 0$.

\end{proof}

\begin{example}\label{ex}
Consider the matrix $A=\begin{pmatrix}2&1\\1&1\end{pmatrix}\in\SL_2(\mZ)$, inducing a diffeomorphism $f_A$ of $\mT^2$ and let $M_A$ denote the mapping torus of $f_A$ as before. Since 
$\frac{3+\sqrt 5}2$ is an eigenvalue of $A$, Lemma \ref{ec} shows that for $c:=\ln{\frac{3+\sqrt 5}2}$, the first twisted cohomology group associated to the nowhere vanishing 1-form $\theta_c:=c\,\di t$ on $M$ is non-zero: $H^1_{\theta_c}(M_A)\ne 0$.
\end{example}

By \cite[Theorem 4.5]{llmp}, the twisted cohomology associated to a closed 1-form which is parallel with respect to some Riemannian metric, vanishes. The above example thus shows the existence of compact manifolds carrying nowhere vanishing closed 1-forms which are not parallel with respect to any Riemannian metric.

Our last example concerns the twisted cohomology on Vaisman manifolds. Recall that a Vaisman manifold is a locally conformally Kähler manifold with parallel Lee form \cite{v}.
The space of harmonic $1$-forms on a compact Vaisman manifold $(M,g,J)$ with Lee form $\theta$ decomposes as follows:
\begin{equation}\label{decv}\mathcal{H}^1(M,g)=\mathrm{span}\{\theta\}\oplus \mathcal{H}^1_0(M,g),
\end{equation} 
where $\mathcal{H}^1_0(M,g)$  is $J$-invariant and consists of harmonic 1-forms pointwise orthogonal to $\theta$ and $J\theta$ (see for instance  \cite[Lemma 5.2]{mmp}). That means that every harmonic $1$-form on $M$ can be written as $\beta=t\theta+\alpha$, with $t\in\mR$ and $\alpha\in \mathcal{H}^1_0(M,g)$. 

By \cite[Lemma 3.3]{mmp19}, every harmonic form $\beta=t\theta+\alpha$ with $t>0$ is the Lee form of a Vaisman metric on $M$. In particular, for every non-vanishing $t$, there exists a metric on $M$ whith respect to which $\beta$ is parallel. By \cite[Theorem 4.5]{llmp}, the twisted cohomology $H^*_{t\theta+\alpha}(M)$ vanishes for all $t\ne 0$ and $\alpha\in \mathcal{H}^1_0(M,g)$. It remains to understand the case where $t=0$, {\em i.e.} the twisted cohomology associated to forms $\alpha\in \mathcal{H}^1_0(M,g)$.

It turns out that there exist Vaisman manifolds $(M,g)$ with $\mathcal{H}^1_0(M,g)\ne 0$, for which $H^*_{\alpha}(M)$ is non-zero for every $\alpha\in \mathcal{H}^1_0(M,g)\setminus \{0\}$.

\begin{example}\label{vaisman} Let $S$ be a compact oriented Riemann surface and let $\pi:N\to S$ be the principal $S^1$-bundle whose first Chern class is the positive generator $e\in H^2(S,\mZ)$. For every Riemannian metric $g_S$ on $S$, the 3-dimensional manifold $N$ carries a Riemannian metric $g_N$ making $\pi$ a Riemannian submersion, and which is Sasakian. Consequently, the Riemannian product $(M,g):=S^1\times (N,g_N)$ is Vaisman. Its Lee form is just the length element of $S^1$, denoted by $\theta=\di t$. 

The Gysin exact sequence associated to the fibration $\pi:N\to S$ reads
$$0\mapsto H_{\rm dR}^1(S)\stackrel{\pi^*}{\longrightarrow} H_{\rm dR}^1(N)\stackrel{\pi_*}{\longrightarrow}H^0_{\rm dR}(S)\stackrel{c_1(N)\wedge}{\longrightarrow} H^2_{\rm dR}(S){\longrightarrow}\cdots.$$
By the choice of $c_1(N)=e$, the last arrow is an isomorphism, thus showing that $\pi^*:H_{\rm dR}^1(S){\to} H_{\rm dR}^1(N)$ is an isomorphism too. Since $\pi:(N,g_N)\to (S,g_S)$ is a Riemannian submersion, we thus have $\pi^*(\mathcal{H}^1(S,g_S))=\mathcal{H}^1(N,g_N)$. 

Moreover, if $p_2:M=S^1\times N\to N$ denotes the projection on the second factor, we clearly have $\mathcal{H}^1(M,g)=\mathrm{span}\{\theta\}\oplus p_2^*(\mathcal{H}^1(N,g_N)).$

Denoting by $p:=\pi\circ p_2$,
the decomposition \eqref{decv} becomes
\begin{equation}\label{decv1}\mathcal{H}^1(M,g)=\mathrm{span}\{\theta\}\oplus p^*(\mathcal{H}^1(S,g_S)).
\end{equation} 

Let $\alpha$ be a non-zero harmonic form in $\mathcal{H}^1(S,g_S)$ and let $\rho\colon \pi_1(S)\to (\mR^*_+,\times)$ be the character of $\pi_1(S)$ associated to $\alpha$, given by Lemma \ref{h1dR}. Clearly, the character of $\pi_1(M)$ associated to $p^*\alpha$ is $\tilde\rho:=\rho\circ p_*$, where $p_*:\pi_1(M)\to\pi_1(S)$ is the induced morphism of the fundamental groups. Note that, since the fibers of $p:M\to S$ are connected, the exact homotopy sequence shows that $p_*$ is surjective.

By the proof of Corollary \ref{rs}, there exists a two-dimensional representation $\xi:\pi_1(S)\to \GL_2(\mR)$ with $\det(\xi)={\rho}$, which fixes the vector $\begin{pmatrix}1\\0\end{pmatrix}\in\mR^2$ and whose restriction to the commutator $[\pi_1(S),\pi_1(S)]$ is non-trivial. 

Composing $\xi$ with $p_*$ yieds a two-dimensional representation $\tilde\xi:=\xi\circ p_*:\pi_1(M)\to \GL_2(\mR)$ with $\det(\tilde\xi)={\tilde\rho}$, which fixes the vector $\begin{pmatrix}1\\0\end{pmatrix}\in\mR^2$ and whose restriction to $[\pi_1(M),\pi_1(M)]$ is non-trivial (since $p_*$ is surjective). By Theorem \ref{pr}, the first twisted cohomology group $H^1_{p^*\alpha}(M)$ is non-vanishing.

\end{example}

\end{document}